\documentclass[11pt]{article}

\usepackage{amssymb, amsmath, amsthm, graphicx}
\usepackage[left=1in,top=1in,right=1in]{geometry}

\usepackage{subfigure}

      \theoremstyle{plain}
      \newtheorem{theorem}{Theorem}[section]
      \newtheorem{lemma}[theorem]{Lemma}

      \theoremstyle{definition}
      \newtheorem{definition}[theorem]{Definition}

      \theoremstyle{remark}

\title{Disjoint edges in complete topological graphs}

\author{Andrew Suk\thanks{Massachusetts Institute of Technology, Cambridge. Email: {\tt asuk@math.mit.edu}.  Research funded by an NSF Postdoctoral Fellowship, and by Swiss National Science Foundation Grant 200021-125287/1. }}

\begin{document}

\maketitle

\begin{abstract}
It is shown that every complete $n$-vertex simple topological graph has at $\Omega(n^{1/3})$ pairwise disjoint edges, and these edges can be found in polynomial time.  This proves a conjecture of Pach and T\'oth, which appears as problem 5 from chapter 9.5 in \emph{Research Problems in Discrete Geometry} by Brass, Moser, and Pach.
\end{abstract}

\section{Introduction}

Given a collection of objects $\mathcal{C}$ in the plane, the \emph{intersection graph} $G(\mathcal{C})$ has vertex set $\mathcal{C}$ and two objects are adjacent if and only if they have a nonempty intersection.  The \emph{independence number} of $G(\mathcal{C})$ is the size of the largest independent set, that is, the size of the largest subfamily of pairwise disjoint objects in $\mathcal{C}$.  It is known that computing the independence number of an intersection graph is NP-hard \cite{garey}, even for very simple objects such as rectangles and disks in the plane \cite{imai, fowler}.  However, due to its applications in VLSI design \cite{vlsi}, map labeling \cite{map}, and elsewhere, a lot of research has been devoted to developing polynomial-time approximation schemes (PTAS) for computing the independence number of intersection graphs (see \cite{foxpach} for more references).  In this paper, we study the independence number of the intersection graph of edges in a complete simple topological graph.

A \emph{topological graph} is a graph drawn in the plane such that its vertices are represented by points
 and its edges are represented by non-self-intersecting arcs connecting the corresponding points. The arcs are
allowed to intersect, but they may not intersect vertices except for
their endpoints.  Furthermore, no three edges have a common interior point and no two edges are tangent, i.e., if two edges share an interior point, then they must properly cross at that point in common.  A topological graph is \emph{simple} if every pair of its edges intersect at most once.  Two edges of a topological graph \emph{cross} if their interiors share a point, and are \emph{disjoint} if they neither share a common vertex nor cross.  If the edges are drawn as straight-line segments, then the graph is \emph{geometric}.

Let $F$ denote the graph obtained from the complete graph on 5 vertices by subdividing each edge with an extra vertex.  It is easy to see that every intersection graph $G$ of curves in the plane does not contain $F$ as an induced subgraph \cite{tarjan}.  By applying a theorem of Erd\H{o}s and Hajnal \cite{hajnal}, every complete $n$-vertex simple topological graph contains $e^{\Omega(\sqrt{\log n})}$ edges that are either pairwise disjoint or pairwise crossing.  However, it was suspected \cite{solymosi} that this bound is far from optimal.  Fox and Pach \cite{pach} showed that there exists a constant $\delta> 0$, such that every complete $n$-vertex simple topological graph contains $\Omega(n^{\delta})$ pairwise crossing edges.  However, much weaker bounds were previously known for pairwise disjoint edges.

In 2003, Pach, Solymosi, and T\'oth \cite{solymosi} showed that every complete $n$-vertex simple topological graph has at least $\Omega(\log^{1/6}n)$ pairwise disjoint edges.  This lower bound was later improved by Pach and T\'oth \cite{toth} to $\Omega(\log n/\log\log n)$.  Recently, Fox and Sudakov \cite{fox} gave a modest improvement of $\Omega(\log^{1 + \epsilon}n)$, where $\epsilon$ is a very small constant.  We note that the previous two bounds hold for dense simple topological graphs.  Pach and T\'oth conjectured (see problem 5 of chapter 9.5 in \cite{brass}) that there exists a constant $\delta>0$ such that every complete $n$-vertex simple topological graph has at least $\Omega(n^{\delta})$ pairwise disjoint edges.  Our main result settles the conjecture in the affirmative.

\begin{theorem}
\label{main}
Every complete $n$-vertex simple topological graph contains $\Omega(n^{1/3})$ pairwise disjoint edges.
\end{theorem}

Note that Theorem \ref{main} does not remain true if the \emph{simple} condition is dropped.  Indeed, in \cite{toth}, Pach and T\'oth gave a construction of a complete $n$-vertex topological graph such that every pair of edges intersect exactly once or twice.

\section{The dual shatter function}

In this section, we will recall one of the most useful parameters measuring the complexity of a set system: the dual shatter function.  All of the following concepts and results can be found in Chapter 5 of \cite{matousek}.  Let $(X,\mathcal{S})$ be a set system with ground set $X$, such that $X$ is finite.

\begin{definition}
The \emph{dual shatter function} of $(X,\mathcal{S})$ is a function, denoted by $\pi^*_\mathcal{S}$, whose value at $m$ is defined as the maximum number of equivalence classes on $X$ defined by an $m$-element subfamily $\mathcal{Y}\subset \mathcal{S}$, where two points $x,y \in X$ are equivalent with respect to $\mathcal{Y}$ if $x$ belongs to the same sets of $\mathcal{Y}$ as $y$ does.  In other words, $\pi^*_{\mathcal{S}}(m)$ is the maximum number of nonempty cells in the Venn diagram of $m$ sets of $\mathcal{S}$.
\end{definition}

One of the main tools used to prove Theorem~\ref{main} is the following result of Chazelle and Welzl on matchings with low stabbing number.  A similar approach was done by Pach in \cite{pach2}, who showed that every $n$-vertex complete geometric graph contains $\Omega(n^{1/2})$ pairwise parallel edges, where two edges are \emph{parallel} if they are the opposite sides of a convex quadrilateral.

 Given a set system $(X,\mathcal{S})$ and a graph $G = (X,E)$, we say that a set $S \in \mathcal{S}$ \emph{stabs} edge $uv \in E(G)$ if $|S\cap \{u,v\}| = 1$.  The \emph{stabbing number of $G$ with respect to the set $S$} is the number of edges of $G$ stabbed by $S$, and the \emph{stabbing number of $G$} is the maximum of stabbing numbers of $G$ with respect to all sets of $\mathcal{S}$.

\begin{lemma}[\cite{welzl}]
\label{match}
Let $n$ be an even integer, and $\mathcal{S}$ be a set system on an $n$-point set $X$ with $\pi^*_{\mathcal{S}}(m) \leq Cm^d$ for all $m$, where $C$ and $d$ are constants.  Then there exists a perfect matching $M$ on $X$ (i.e. a set of $n/2$ vertex disjoint edges) whose stabbing number is at most $C_1n^{1 - 1/d}$, where $C_1$ is a constant that depends only on $C$ and $d$.

\end{lemma}

We will also need the following two lemmas on the dual shatter function of a set system.

\begin{lemma}
\label{union}
Let $\mathcal{S}_1,\mathcal{S}_2$ be set systems on a set $X$ and let $\mathcal{S} = \mathcal{S}_1\cup \mathcal{S}_2$.  Then $\pi_\mathcal{S}^*(m) \leq \pi_{\mathcal{S}_1}^*(m)\pi_{\mathcal{S}_2}^*(m)$.

\end{lemma}

\begin{proof}

Suppose that $\pi^*_{\mathcal{S}}(m)$ is maximized for some $m_1$ sets in $\mathcal{S}_1$ and $m_2$ sets in $\mathcal{S}_2$ where $m_1 + m_2 = m$.  Then the number of nonempty cells in the Venn diagram of these $m$ sets is at most

$$\pi^*_{\mathcal{S}_1}(m_1)\pi^*_{\mathcal{S}_2}(m_2) \leq \pi^*_{\mathcal{S}_1}(m)\pi^*_{\mathcal{S}_2}(m).$$

\end{proof}

\begin{lemma}
\label{expression}
Let $\Phi(X_1,X_2,...,X_t)$ be a fixed set-theoretic expression (using the operations of union, intersection, and difference) with variables $X_1,...,X_t$ standing for sets.  Let $\mathcal{S}$ be a set system on a set $X$.  Let $\mathcal{T}$ consist of all sets $\Phi(S_1,...,S_t)$ for all possible choices $S_1,...,S_t \in \mathcal{S}$.  If $\pi^*_\mathcal{S}(m) \leq Cm^d$ for every $m$, then $\pi^*_\mathcal{T}(m) \leq Ct^dm^d$.

\end{lemma}

\begin{proof}
Each set in $\mathcal{T}$ is a disjoint union of some cells of the Venn diagram of some $t$ sets in $\mathcal{S}$.  Therefore, given any $m$ sets in $\mathcal{T}$, each cell in the Venn diagram of these $m$ sets is a disjoint union of some cells in the Venn diagram of at most $tm$ sets of $\mathcal{S}$.  Hence

$$\pi^*_\mathcal{T}(m) \leq \pi^*_\mathcal{S}(tm) \leq Ct^dm^d.$$
\end{proof}

\section{Proof of Theorem~\ref{main}}

Let $K = (V,E)$ be a complete simple topological graph with $n+1$ vertices.  Without loss of generality, we can assume $n$ is even.  Notice that the edges of $K$ divide the plane into several cells (regions), one of which is unbounded.  We can assume that there is a vertex $v_0 \in V$ such that $v_0$ lies on the boundary of the unbounded cell.  Indeed, otherwise we can project $K$ onto a sphere, then choose an arbitrary vertex $v_0$ and then project $K$ back to the plane such that $v_0$ lies on the boundary of the unbounded cell, and moreover two edges cross (are disjoint) in the new drawing if and only if they crossed (were disjoint) in the original drawing.

Given a subset $V' \subset V$, we denote by $K[V']$ the topological subgraph induced by the vertex set $V'$.  Since we will be dealing with both combinatorial and topological graphs, the term \emph{edge} will refer to an edge in a combinatorial graph and will be denoted by pairs of vertices $v_iv_j$.  The term \emph{topological edge} will refer to edges in the topological graph $K$, and will be denoted by $K[v_i,v_j]$ (i.e. the topological subgraph induced by vertices $v_i$ and $v_j$).

Consider the topological edges emanating out from $v_0$, and label their endpoints $v_1,...,v_n$ in counterclockwise order so that for $i < j$, the vertices $(v_0,v_i,v_j)$ appear in counterclockwise order along the simple closed curve $K[v_0,v_i,v_j]$.  This is possible since $v_0$ lies on the unbounded cell and these three edges do not cross.  For all pairs $i,j$, we will call the topological subgraph $K[v_0,v_i,v_j]$ a \emph{triangle} in $K$.  Note that each ``triangle" is incident to $v_0$, and is a simple closed curve since $K$ is simple.

Set $X = \{v_1,...,v_n\}$, and now we will define the set system $(\mathcal{S}_1,X)$ as follows.  For each pair $i,j$ that satisfies $1 \leq i < j \leq n$, we define the set $S_{i,j}$ to be the set of vertices in $X$ that lie in the interior of triangle $K[v_0,v_i,v_j]$.  See Figure~\ref{set1} for a small example. Then set

$$\mathcal{S}_1 = \bigcup\limits_{1 \leq i < j \leq n} \left\{S_{i,j}\right\}.$$

\begin{lemma}
\label{s1}
Let $(\mathcal{S}_1,X)$ be the set system defined as above.  Then $\pi^*_{\mathcal{S}_1}(m) \leq 5m^2$.

\end{lemma}

\begin{proof}
In order to obtain an upper bound on $\pi^*_{\mathcal{S}_1}(m)$, we simply need to bound the maximum number of regions for which $m$ triangles in $K$ partitions the plane into.  We proceed by induction on $m$.  The base case when $m = 1$ is trivial.  Now assume that the statement holds up to $m-1$.  Notice that any set of $m-1$ triangles consists of at most $3(m-1)$ topological edges of $K$.  Therefore, when we add the $m$th triangle $T$, each topological edge in $T$ creates at most $3(m-1)$ new crossing points (since $K$ is simple).  Hence triangle $T$ creates at most $9(m-1)$ new regions in the arrangement.  By the induction hypothesis, we have at most

$$5(m-1)^2 + 9(m-1) \leq 5m^2$$

\noindent cells created by the $m$ triangles.

\end{proof}

Now we define the set system $(\mathcal{S}_2,X)$ as follows.  For each pair $i,j$ that satisfies $1\leq i < j \leq n$, we define the set $S'_{i,j} = \{v_k\in X : K[v_0,v_k]\textnormal{ crosses }K[v_i,v_j]\}$.  See Figure~\ref{set2} for a small example.  Then set

$$\mathcal{S}_2 = \bigcup\limits_{1\leq i < j \leq n}\left\{S'_{i,j}\right\}.$$

 \begin{figure}
  \centering
  \subfigure[$S_{3,6} = \{v_{1},v_4\}$.]{\label{set1}\includegraphics[width=0.3\textwidth]{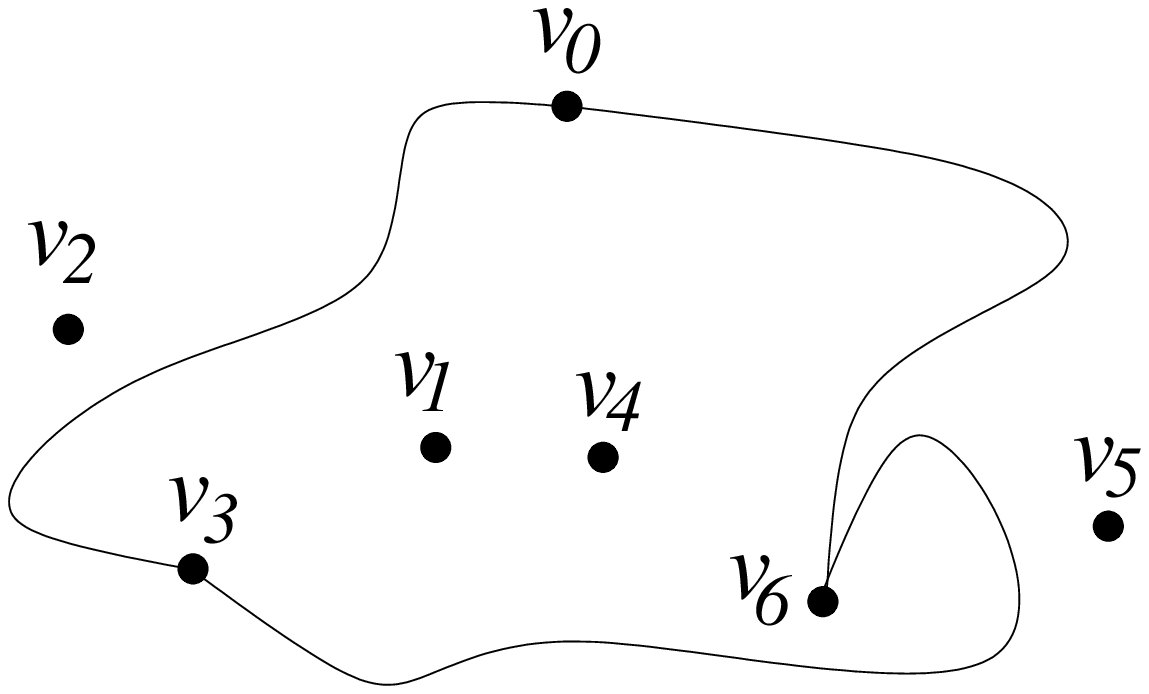}}\hspace{2cm}
\subfigure[$S'_{3,6} = \{v_{1},v_5\}$.]{\label{set2}\includegraphics[width=0.3\textwidth]{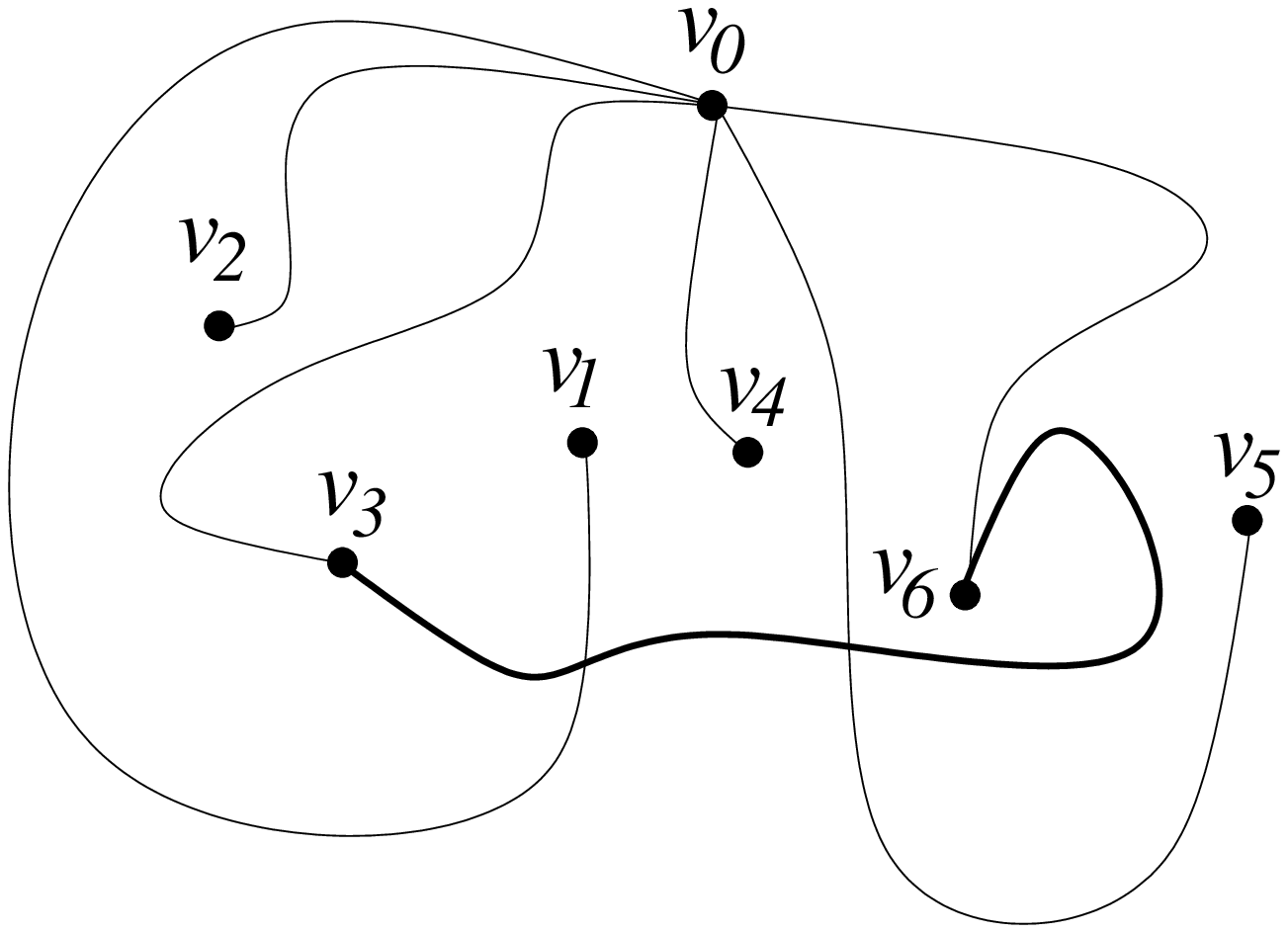}}
 \caption{Sets $S_{3,6}\in \mathcal{S}_1$ and $S'_{3,6}\in \mathcal{S}_2$.}
\end{figure}

\begin{lemma}
Let $\mathcal{S} = \mathcal{S}_1 \cup \mathcal{S}_2$ where $\mathcal{S}_1,\mathcal{S}_2$ are defined as above.  Then $\pi^*_{\mathcal{S}}(m) \leq 120m^3$.

\end{lemma}

\begin{proof}
Let $i,j$ be fixed such that $i < j$.  Notice that if $k$ is an integer that satisfies $k < i$ or $k > j$, then $v_k \in S'_{i,j} \in \mathcal{S}_2$ if and only if $v_k \in S_{i,j} \in \mathcal{S}_1$.  Moreover, if $k$ is an integer that satisfies $i < k < j$, then $v_k \in S'_{i,j} \in \mathcal{S}_2$ if and only if $v_k \not\in S_{i,j} \in \mathcal{S}_1$.  In what follows, we will apply Lemmas~\ref{union} and \ref{expression} in order to obtain the set system $\mathcal{S} = \mathcal{S}_1\cup \mathcal{S}_2$.  This will ensure that the dual shatter function $\pi^*_{\mathcal{S}}$ is \emph{well behaved}.

Let $(\mathcal{I},X)$ be a set system such that

$$\mathcal{I} = \bigcup\limits_{1 \leq i < j \leq n} \left\{I_{ij}\right\},$$

\noindent where $I_{ij} = \{v_k : i < k < j\}$.  It is easy to see that $\pi^*_{\mathcal{I}}(m) \leq 2m+1\leq 3m$, since this is equivalent to determining the number of cells created by $m$ intervals on the real line. Let $\mathcal{T}_1$ be the set system defined as $\mathcal{T}_1 = \mathcal{S}_1\cup \mathcal{I}\cup \{\emptyset\}$.  By Lemmas~\ref{union} and \ref{s1}, we have $\pi^*_{\mathcal{T}_1}(m) \leq 15m^3$.  Notice that each set $S'_{ij} \in \mathcal{S}_2$ can be expressed as

$$S'_{ij} = (S_{ij} \setminus I_{ij}) \cup (I_{ij}\setminus S_{ij})$$

\noindent where $S_{ij} \in \mathcal{S}_1$ and $I_{ij} \in \mathcal{I}$ are defined above.  Therefore, let $\Phi(X_1,X_2)$ be the set-theoretic expression defined as

$$\Phi(X_1,X_2) = (X_1\setminus X_2) \cup (X_2\setminus X_1),$$

\noindent and let $\mathcal{T}_2$ consist of all sets $\Phi(X_1,X_2)$ for all possible choices $X_1,X_2 \in \mathcal{T}_1$.  Hence $\mathcal{S} = \mathcal{S}_1\cup \mathcal{S}_2 \subset \mathcal{T}_2$.  By Lemma~\ref{expression}, we have

$$\pi^*_{\mathcal{S}}(m) \leq \pi^*_{\mathcal{T}_2}(m)\leq \pi^*_{\mathcal{T}_1}(2m) \leq 15(2m)^3   = 120m^3.$$

\end{proof}

By setting $\mathcal{S} = \mathcal{S}_1\cup \mathcal{S}_2$, we know that $\pi^*_{\mathcal{S}}(m) \leq 120m^3$.  Therefore we can apply Lemma~\ref{match} to $(\mathcal{S},X)$ and find a perfect matching $M$ on $X$ whose stabbing number is at most $Cn^{2/3}$, where $C$ is an absolute constant.  In other words, each set in $S \in \mathcal{S}$ stabs at most $Cn^{2/3}$ members in $M$.

Now let $G = (V,E)$ be an $(n/2)$-vertex graph, where $V(G) = M$, and two vertices $v_{i}v_j$, $v_kv_l$ in $G$ (edges in $M$) are adjacent if and only if

\begin{enumerate}

\item $S_{i,j} \in \mathcal{S}_{1}$ stabs $v_kv_l$, or

\item $S'_{i,j} \in \mathcal{S}_2$ stabs $v_kv_l$, or

\item $S_{k,l} \in  \mathcal{S}_1 $ stabs $v_iv_j$, or

\item $S'_{k,l} \in \mathcal{S}_2$ stabs $v_iv_j$.

\end{enumerate}

Since $M$ has a stabbing number at most $Cn^{2/3}$, the two sets corresponding to each vertex in $G$ stab (in total) at most $2\cdot Cn^{2/3}$ other edges in $M$.   A simple counting argument shows that $G$ contains at most $Cn^{5/3}$ edges.  Therefore we can use the following well known theorem of Tur\'an to find a large independent set in $M$.

\begin{theorem}
\label{turan}
Every graph with $n$ vertices and $e$ edges contains an independent set of size at least $n^2/(2e + n)$.

\end{theorem}

In particular, the graph $G$ defined above contains an independent set of size $\Omega(n^{1/3})$.  Let $M'\subset M$ be the edges corresponding to this independent set of size $\Omega(n^{1/3})$.  Now we claim that the corresponding topological edges in $K$ are pairwise disjoint.  Indeed, let $e_1,e_2 \in M'$ such that $e_1 = v_iv_j$ and $e_2 = v_kv_l$.  Since the set $S_{i,j} \in \mathcal{S}_1$ does not stab $v_kv_l$, both vertices $v_k,v_l$ must either lie inside or outside of triangle $K[v_0,v_i,v_j]$.  See Figure~\ref{case1}.

\medskip

\noindent \emph{Case 1.}  Suppose that $v_k,v_l$ lie outside of triangle $K[v_0,v_i,v_j]$.  For sake of contradiction, suppose that the topological edges $K[v_i,v_j]$ and $K[v_k,v_l]$ cross.  By a simple parity argument, $K[v_k,v_l]$ must cross either $K[v_0,v_i]$ or $K[v_0,v_j]$, but not both.  However, this implies that the set $S'_{k,l}\in \mathcal{S}_2$ stabs $v_iv_j$ which is a contradiction.  Therefore, $K[v_i,v_j]$ and $K[v_k,v_l]$ must be disjoint.

\medskip

 \begin{figure}[h]
  \centering
  \subfigure{\label{case1}\includegraphics[width=0.3\textwidth]{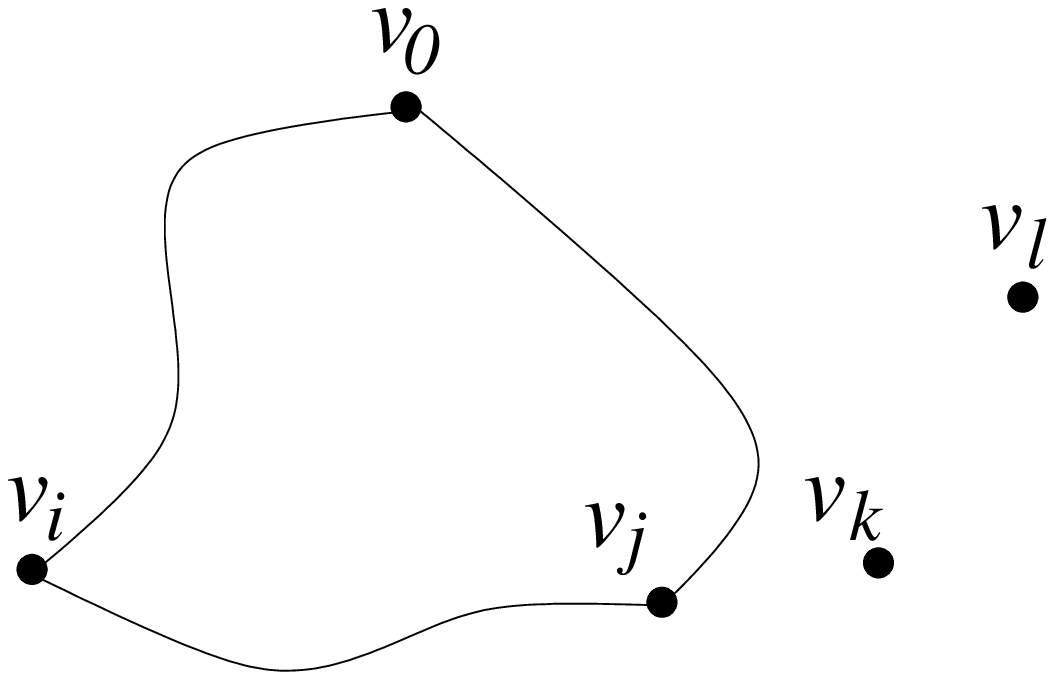}}\hspace{2cm}
\subfigure{\label{case1b}\includegraphics[width=0.22\textwidth]{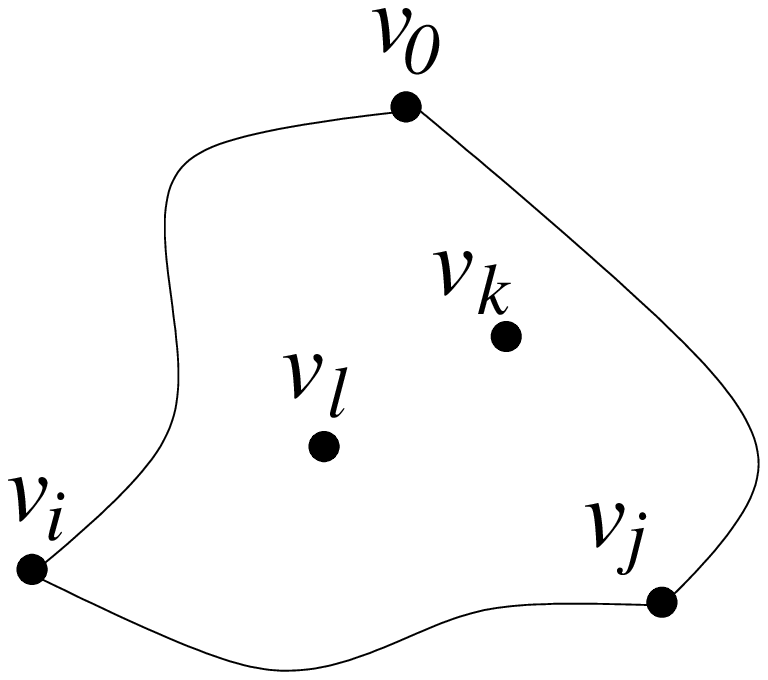}}
 \caption{Case 1 and 2.}
  \label{case1}
\end{figure}

\noindent \emph{Case 2.}  If $v_k,v_l$ lie inside of triangle $K[v_0,v_i,v_j]$, then by the same parity argument as above, $K[v_i,v_j]$ and $K[v_k,v_l]$ must be disjoint.

\section{Concluding Remarks}

\begin{enumerate}

\item \noindent \textbf{Remarks on algorithms.}  Since the matching in Lemma \ref{match} and the independent set in Theorem \ref{turan} can be found in polynomial time (see \cite{welzl} and \cite{spencer}), the proof above gives a polynomial time algorithm for finding $\Omega(n^{1/3})$ pairwise disjoint edges in a complete $n$-vertex simple topological graph.

    \item We proved that every complete $n$-vertex simple topological graph has at least $\Omega(n^{1/3})$ pairwise disjoint edges.  It would be interesting to see if a polynomial bound holds for dense simple topological graphs.

\item We proved that the dual shatter function $\pi^{\ast}_{\mathcal{F}}(m) = O(m^3)$.  Recently G\'eza T\'oth \cite{geza} gave a construction showing that this bound is tight.

\item To our knowledge, there are no constructions of a complete $n$-vertex simple topological graph with sublinear number of pairwise disjoint edges.  However, let us remark that Valtr \cite{brass} gave a construction of a complete $n$-vertex simple topological graph, such that every edge crosses at least $\Omega(n^{3/2})$ other edges.

\end{enumerate}

\end{document}